\documentclass[12pt]{amsart}

\usepackage[mathlines]{lineno}
\modulolinenumbers[1]

\usepackage[colorlinks,urlcolor=black,citecolor=black,linkcolor=black]{hyperref}
\usepackage{amsthm}
\usepackage{amssymb}
\usepackage{amsmath}
\usepackage{enumerate}
\usepackage{mathrsfs}
\usepackage{pdfrender}
\usepackage{multicol}

\newcommand{\nint}{\int_{-\infty}^{+\infty}}
\newcommand{\R}{\mathbb{R}}

\newcommand{\C}{\mathbb{C}}

\newcommand{\E}{{E}}

\renewcommand{\Im}{\mathrm{Im}}
\def\MR#1{\href{http://www.ams.org/mathscinet-getitem?mr=#1}{MR#1}}

\theoremstyle{definition}

\newtheorem{lemma}{Lemma}[]
\newtheorem{remark}{Remark}[]
\newtheorem{theorem}{Theorem}[]

\newtheorem{proposition}{Proposition}

\newtheorem*{definition*}{Definition}

\begin{document}
\thispagestyle{empty}
\setlength{\itemindent}{0em}
\setlength{\itemsep}{0.3em}

\title[Stability in the quadratic-cubic Klein-Gordon equation]{Stability and instability of standing-wave solutions to one dimensional quadratic-cubic Klein-Gordon equations}

\author[D. Garrisi]{Daniele Garrisi\(^{*}\)}

\address{
Room 324, Sir Peter Mansfield Building\\
School of Mathematical Sciences\\
University of Nottingham Ningbo China\\
199 Taikang East Road\\
315100, Ningbo, China}

\email{daniele.garrisi@nottingham.edu.cn}

\thanks{%
This work has been supported by the \textsl{PDE Research Group} of School of Mathematical Sciences of the University of Nottingham Ningbo China and funded by the \textsl{FoSE New Researchers Grant}.}



\dedicatory{}

\begin{abstract}
We study the stability of standing-waves solutions to a scalar
non-linear Klein-Gordon equation in dimension one with a quadratic-cubic
non-linearity. Orbits are obtained by applying the semigroup 
generated by the negative complex unit multiplication on a critical
point of the energy constrained to the charge.
\end{abstract}
\maketitle
\section*{Introduction}
This work aims to classify the stability and instability of standing-wave
solutions to the non-linear Klein-Gordon equation which can be written in the form
\begin{linenomath}\begin{equation}
\label{eq.ham}
\frac{du}{dt} = JE'(u(t)),
\end{equation}\end{linenomath}
where \(E\) is functional defined on \(X_r := H^1 _r (\mathbb{R};\mathbb{C})\times
L^2 _r (\mathbb{R};\mathbb{C})\), the space of complex-valued functions which are
radially symmetric with respect to the origin, and \(J\) is an unbounded linear
transformation \(J\colon X_r^*\to X_r\). We will apply the abstract method devised 
in \cite{GSS87,GSS90}, and in \cite{Sha83,SS85}. We define the set
\begin{linenomath}\begin{equation}
\label{eq.orbit}
Orb(u_\omega) := \{T(s\omega)u_\omega\mid s\in\mathbb{R}\}
\end{equation}\end{linenomath}
where \(T(s)\) is a continous semi-group of operators acting on \(X_r\).
The equation to which we would like to apply this method is the 
quadratic-cubic one-dimensional non-linear Klein-Gordon equation
\begin{linenomath}\begin{equation}
\tag{NLKG}
\label{eq.NLKG}
(\partial_{tt}^2 - \partial_{xx}^2 + m^2)\phi - 3a|\phi|\phi + 
4b|\phi|^2\phi = 0,
\end{equation}\end{linenomath}
where \(m,a\) and \(b\) are positive real numbers. 
A standing-wave is a solution to \eqref{eq.NLKG} which can be written as
\begin{linenomath}\begin{equation*}
\phi(t,x) := e^{-i\omega t} R(x),\quad (t,x)\in\R\times\R
\end{equation*}\end{linenomath}
where $ \omega $ is a real number and $ R $ is a real-valued function of class
$ H^1 _r (\R;\R) $. Dispersive equations with competing powers non-linearities have been proposed for several applications. In \cite{BGM89} a one-dimensional cubic-quintic non-linear Schr\"odinger equation arises from boson gas interaction. The non-linear Klein-Gordon equation models the field equation for
spin-0 particles, \cite[\S2]{Lee81}; for a cubic-quintic non-linear Klein-Gordon
equation, \cite{Sha83} proved that there are stable and unstable standing-waves when \(\omega\) gets close to
\(13/16\) and \(1\), respectively, a result anticipated by the numerical inspection 
in \cite{And71}. We also quote the work of G.~Rosen, \cite{Ros65}, for a quintic non-linearity in 
dimension three.

Standing-waves can be obtained as 
minima of the energy functional 
\begin{linenomath}\begin{equation*}
\begin{split}
\E(\phi,\psi) :&= \frac{1}{2}\nint |\psi(x)|^2 dx + 
\frac{1}{2}\nint|\phi'(x)|^2 dx \\
&+ \frac{m^2}{2}\nint|\phi(x)|^2 dx
+\nint G(|\phi(x)|)dx
\end{split}
\end{equation*}\end{linenomath}
on the constraint
\begin{linenomath}\begin{equation}
\label{eq.q}
\begin{split}
Q(\phi,\psi) := -\Im\nint \psi(x)\overline{\phi(x)} dx = (\phi,i\psi)_2,
\end{split}
\end{equation}\end{linenomath}
which is the method followed in \cite{BBBM10}. However, existence of standing-waves will be proved by solving an initial value problem, using the approach
devised in \cite{IK93,Mae08,GG17} and properties of one-dimensional elliptic equations proved in \cite{BL83a}. In this way, we will be able to discuss
the stability of all the standing-waves, not just the ones arising as minima
of \(E\) constrained on \(Q\).
To prove the stability of the set 
\eqref{eq.orbit}, we will check Assumptions 1, 2 and 3 presented in the introduction
of \cite{GSS87}, construct a smooth one-parameter family \(u_\omega\) of solutions to the differential equation
\begin{linenomath}\begin{equation*}
E'(u_\omega)=\omega Q'(u_\omega)
\end{equation*}\end{linenomath}
and study the convexity of the function 
\begin{linenomath}\begin{equation*}
d(u_\omega) := E(u_\omega)-\omega Q(u_\omega).
\end{equation*}\end{linenomath}
The conclusions of the first three chapters hold as long as \(G\) is \(C^3\).
Considering a quadratic-cubic non-linearity, as we will do in Chapter~\ref{ch.4},
allows to find an exact form for \(d''(\omega)\).
We introduce the following real-valued function defined on \((0,+\infty)^3\)
\begin{linenomath}\begin{equation*}
\tau(a,b,m) := \frac{2m^2 b}{a^2}
\end{equation*}\end{linenomath}
which is convenient to classify the stability of all the standing-waves
for every \(m > 0\) and non-linearity \(-as^2 + bs^3\) with \(a,b > 0\).
\begin{theorem}[Stability and instability]
\label{thm.stability-and-instability}
There exist \(\tau_* > 2\) such that 
\begin{enumerate}[(i). ]
	\item if \(\tau(a,b,m)\geq\tau_*\), then \(Orb(u_\omega)\) is stable for every 
	\(\omega\in(\omega^*,m)\)
	\item if \(1 < \tau(a,b,m) < \tau_*\), there are \(\omega_1(\tau)\) and \(\omega_2(\tau)\) 
	such that \(Orb(u_\omega)\) is stable for
	 \begin{linenomath}\begin{equation*}
	 \omega\in (\omega_*,\omega_1(\tau))\cup(\omega_2(\tau),m)
	 \end{equation*}\end{linenomath} 
	 and unstable for 
	 \begin{linenomath}\begin{equation*}
	 \omega_1(\tau)\leq\omega\leq \omega_2(\tau).
	 \end{equation*}\end{linenomath}
\item if \(0 < \tau(a,b,m) \leq 1\), there exists \(\omega_3(\tau)\) such that 
\(Orb(u_\omega)\) is stable for 
\begin{linenomath}\begin{equation*}
\omega\in (\omega_*,\omega_3(\tau))
\end{equation*}\end{linenomath}
and unstable for 
\begin{linenomath}\begin{equation*}
\omega\in (\omega_3(\tau),m).
\end{equation*}\end{linenomath}
\end{enumerate}
\end{theorem}
The constant \(\tau_*\) is an irrational number as it will be clear
from the construction. Numerical approximations show that 
\(\tau_*\in (1.13,1.14)\), while the proof that it is bigger than 1 
follows from simple properties of an auxiliary function (labelled \(k_2\) in
the proof of the theorem).
Though different differential equations and systems are studied, for the technique deployed other references as \cite[Examples A,C,D,E]{GSS87},
\cite{Oht95,Mae08} and \cite{ACF14}
are worth to be mentioned. In particular, 
\cite{Oht95} addresses the quadratic-cubic non-linear Schr\"odinger equation, in the case where \(a < 0\) and \(b > 0\). For the sake of clarity, we remark that the set in \eqref{eq.orbit} is different from the \textit{ground-state},
defined as
\begin{linenomath}\begin{equation*}
\Gamma := \{e^{-is\omega} u_\omega (\cdot - y)\mid (s,y)\in\R^2\}
\end{equation*}\end{linenomath}
which includes argument translation together with a phase change. The stability of the ground-state
for the non-linear Klein-Gordon equation and systems has been studied in \cite{BBBM10,Gar12}, 
while we refer to \cite{BBGM07,GJ16,LNW16,LZ21} for the non-linear Schr\"odinger equation and systems. In these works, 
the Concentration-Compactness Lemma of \cite{Lio84a,Lio84b} is an essential ingredient.
Other works we would like to include are \cite{BWW07,BJS16,Cor16a,NTV14}.
\section{Preliminary notations}
In this section, we write the non-linear Klein-Gordon equation in the Hamiltonian
form \eqref{eq.ham} and check that the assumptions made in the introduction of
\cite{GSS87} are verified. We follow a similar scheme to \cite[Example~A,~\S6]{GSS87}, where the authors proved that traveling waves are not stable. 
We set
\begin{linenomath}\begin{equation}
\label{eq.cubic-quartic}
G(s) := -as^3 + bs^4,\quad F(s)=\frac{1}{2} m^2 s^2 + G(s)
\end{equation}\end{linenomath}
for every \(s\geq 0\).
Given \(u_1 = (\phi_1,\psi_1)\) and
\(u_2 = (\phi_2,\psi_2)\) in \(X_r\) we define the
inner product
\begin{linenomath}\begin{equation*}
(u_1,u_2)_{X_r} := m^2(\phi_1,\phi_2)_2 + (\phi_1',\phi_2')_2 + (\psi_1,\psi_2)_2
\end{equation*}\end{linenomath}
where
\begin{linenomath}\begin{equation*}
(f,g)_2 := \text{Re}\nint f(x)\overline{g(x)}dx.
\end{equation*}\end{linenomath}
We use the notation
\begin{linenomath}
\begin{equation}
\label{eq.Riesz}
I\colon X_r\to X_r^*
\end{equation}
\end{linenomath}
for the Riesz isomorphism between real Hilbert spaces.
We use \(\langle \xi,u\rangle\) to denote \(\xi(u)\) for every \(\xi\in X_r^*\) and \(u\in X_r\). We denote by 
\(\mathscr{L}(X,Y)\) and \(\mathscr{L}_c (X,Y)\) the space of bounded and compact operators between two Hilbert spaces \(X\) and \(Y\), and
by \(GL(X)\) the group of invertible bounded operators on \(X\). By \(C^k_0(\R,\C)\) we denote the space of continuously differentiable (to the order \(k\)) functions
with derivative vanishing at infinity up to the order \(k\). We will also use the notation \(\mathscr{R}(\zeta)\) for the
resolvent of \(-\Delta\colon H^2\subseteq L^2\to L^2\), for
every \(\zeta\) in the complement of the spectrum of \(-\Delta\).

Let \(I_0\colon L^2_r\times L^2_r\to X_r^*\) be the bounded, injective linear transformation defined as 
\begin{linenomath}\begin{equation}
\label{eq.shatah-identification}
\langle I_0(u),(h,k)\rangle = (\phi,h)_2 + (\psi,k)_2.
\end{equation}\end{linenomath}
for every \(u := (\phi,\psi)\in X_r\) and \((h,k)\in X_r\). In the rest of
this section, we introduce notations for operators and semigroups of operators
necessary to the abstract framework devised in \cite{GSS87}, and check that they satisfy assumptions required therein.
\paragraph*{The operator \(J\)}
Let \(J\colon D(J)\subseteq X_r^*\to X_r\) be the unbounded operator with 
dense domain \(D(J) := I_0(L^2 _r\times H^1 _r)\)
such that 
\(
JI_0(\psi,\phi) = (\phi,-\psi).
\)
\begin{proposition}
\(J\) is closed, onto and skew-adjoint. That is, \(\langle \xi,J\eta\rangle = - \langle \eta,J\xi\rangle\) for every \(\xi,\eta\in D(J)\subseteq X_r^*\).
\end{proposition}
\begin{proof}
Let \((\xi_n,J\xi_n)\) be a sequence converging to \((\xi,u)\in X_r^*\times X_r\).
Set \(u := (\alpha,\beta)\). There exists a sequence \((\psi_n,\phi_n)\in L^2_r\times H^1_r\) such that
\(I_0 (\psi_n,\phi_n) = \xi_n\) and 
\(\lim_{n\to\infty} (\phi_n,-\psi_n) = (\alpha,\beta)\). For every \((h,k)\in X_r\) we have \((\psi_n,h)_2 + (\phi_n,k)_2 = \langle\xi_n,(h,k)\rangle\).
Taking the limit, we obtain \(\xi=I_0(-\beta,\alpha)\), 
proving that \(\xi\in D(J)\). Also, \(J\xi = u\). Therefore, \(J\) is closed.
Given \(\xi = I_0(\psi_1,\phi_1)\) and \(\eta = I_0 (\psi_2,\phi_2)\), there holds 
\begin{linenomath}\begin{equation*}
\begin{split}
\langle \xi,J\eta\rangle &= \langle \xi,JI_0(\psi_2,\phi_2)\rangle = \langle \xi,(\phi_2,-\psi_2)\rangle \\
&= \langle I_0 (\psi_1,\phi_1),(\phi_2,-\psi_2)\rangle = (\psi_1,\phi_2)_2 - (\phi_1,\psi_2)_2 \\
&= -\big[(\psi_2,\phi_1)_2 - (\phi_2,\psi_1)_2\big] = - \langle I_0 (\psi_2,\phi_2),(\phi_1,-\psi_1)\rangle \\
&= 
- \langle I_0 (\psi_2,\phi_2),J I_0 (\psi_1,\phi_1)\rangle = -
\langle \eta,J\xi\rangle.
\end{split}
\end{equation*}\end{linenomath}
Finally, \(J\) is onto. Given \(u=(\phi,\psi)\in H^1_r\times L^2_r\), 
\(I(-\psi,\phi)\in D(J)\) and
\(JI(-\psi,\phi) = u\).
\end{proof}
\paragraph*{The semigroup \(T\)}
We define \(T(s)\colon X_r\to X_r\) as
\(T(s)(\phi,\psi) = e^{-is}(\phi,\psi)\)
for every \(s\in\mathbb{R}\).
\begin{proposition}
\(T(s)J=JT^*(-s)\) for every \(s\in\R\).
\end{proposition}
\begin{proof}
The two unbounded operators have the same domain \(D(J)\). Given \(\xi\in D(J)\), there exists
\((\psi,\phi)\in L^2_r\times H^1_r\) such that \(\xi = I_0(\psi,\phi)\). Then
\begin{linenomath}\begin{equation*}
\begin{split}
\langle T^*(-s)\xi,(h,k)\rangle &= \langle \xi, (e^{is}h,e^{is}k)\rangle \\
&= 
(\psi,e^{is}h)_2 + (\phi,e^{is}k)_2 = (e^{-is}\psi,h)_2 + (e^{-is}\phi,k)_2.
\end{split}
\end{equation*}\end{linenomath}
Therefore 
\begin{linenomath}\begin{equation*}
\begin{split}
JT^*(-s)I_0 (\psi,\phi) &= JI_0(e^{-is}\psi,e^{-is}\phi) \\
&= (e^{-is}\phi,-e^{is}\psi) =
e^{-is}(\phi,-\psi) = T(s)JI_0(\psi,\phi).
\end{split}
\end{equation*}\end{linenomath}
\end{proof}
\paragraph*{The operator \(B\)}
We define the bounded operator \(B\colon X_r\to X_r^*\) as \(B(u) := I_0 (i\psi,-i\phi)\).
\begin{proposition}
For every \(u,v\in X_r\), there holds 
\begin{enumerate}[(i).]
\item \(Q(u) = \frac{1}{2}\langle Bu,u\rangle\)
\item \(JB(u) = T'(0)u\)
\item \(\langle Bu,v\rangle = \langle Bv,u\rangle\).
\end{enumerate}
\end{proposition}
\begin{proof}
(i).
\begin{linenomath}\begin{equation*}
\begin{split}
\frac{1}{2}\langle Bu,u\rangle &= \frac{1}{2}\langle B(\phi,\psi),(\phi,\psi)\rangle
= \frac{1}{2} (i\psi,\phi)_2 - \frac{1}{2} (i\phi,\psi)_2\\
&= 
\frac{1}{2}\text{Re}\nint i\psi\overline{\phi}dx + \frac{1}{2}\text{Re}\nint -i\phi\overline{\psi}dx\\
&=-\frac{1}{2}\text{Im}\nint\psi\overline{\phi}dx + \frac{1}{2}\text{Re}\nint \overline{-i\phi}\psi dx\\
&=-\frac{1}{2}\text{Im}\nint\psi\overline{\phi}dx + \frac{1}{2}\text{Re}\nint i\overline{\phi}\psi dx\\
&=-\frac{1}{2}\text{Im}\nint\psi\overline{\phi}dx - \frac{1}{2}\text{Im}\nint\overline{\phi}\psi dx = Q(u).
\end{split}
\end{equation*}\end{linenomath}
(ii). \(T'(0)\in\mathscr{L}(X_r,X_r)\) is the infinitesimal generator of the semigroup \(T\) and \(T'(0)u = -iu\). We obtain the equality between the
two operators from
\begin{linenomath}\begin{equation*}
JB(u) = JI_0(i\psi,-i\phi) = (-i\phi,-i\psi) = T'(0)u.
\end{equation*}\end{linenomath}
(iii). We set \(u := (\phi_1,\psi_1)\) and \(v := (\phi_2,\psi_2)\). Then
\begin{linenomath}\begin{equation*}
\begin{split}
\langle Bu,v\rangle &= (i\psi_1,\phi_2)_2 - (i\phi_1,\psi_2)_2 = -(\psi_1,i\phi_2)_2 + (\phi_1,i\psi_2)_2 \\
&= (i\psi_2,\phi_1)_2 -(i\phi_2,\psi_1)_2 = \langle Bv,u\rangle.
\end{split}
\end{equation*}\end{linenomath}
\end{proof}
\section{Regularity assumptions}
In this section, we check the regularity assumptions listed in the introduction of \cite{GSS87}, and
we construct a one-parameter family \(u_\omega\) of solutions
to \(E'(u_\omega)=\omega Q'(u_\omega)\). The whole construction relies on results of
differentiability of Nemitski operators, and compactness of certain linear operators.
We summarize there results in the next proposition. 
\begin{proposition}\label{prop.nem}
\

\begin{enumerate}[(i).]
\itemsep=0.3em
\itemindent=0em
\item Given \(g\) in \(C^1 (\mathbb{R})\) such that \(g(0) = 0\),
the map \(\phi\to g\circ\phi\) is \(C(H^1 _r,H^1_r)\)
\item if \(g\) is \(C^2(\mathbb{R})\) and \(g(0)=g'(0) = 0\), then 
\(\phi\to g\circ\phi\) is \(C^1 (H^1_r,H^1_r)\) and
\begin{linenomath}\begin{equation*}
\langle\mathscr{F} '(\phi_0),\phi\rangle = g'(\phi_0)\phi
\end{equation*}\end{linenomath}
where \(\mathscr{F} (\phi) := g\circ\phi\).
\item Given \(a\in C^1 _0 (\mathbb{R},\mathbb{R})\), the
linear operator mapping \(\phi\) to \(\mathscr{R}(\lambda) (a(x)\phi(x))\) is \(\mathscr{L}_c (H^1_r,H^1_r)\)
for every \(\lambda < 0\)
\item 
\label{item.prop.nem.4}
\(E\) and \(Q\) are \(C^2 (H^1_r (\mathbb{R},\mathbb{C}),\mathbb{R})\). Moreover, given \(u_0 := (\phi_0,\psi_0)\)
\begin{linenomath}
\begin{equation*}
\begin{array}{c}
\langle E'(u_0),(\phi,\psi)\rangle = (\phi_0',\phi')_2 + (\psi_0,\psi)_2 
+ \left(G'(|\phi_0|)\frac{\phi_0 (x)}{|\phi_0(x)|},\phi\right)_2\\[0.7em]
\langle Q'(u_0),(\phi,\psi)\rangle = (\phi_0,i\psi)_2 + (\phi,i\psi_0)_2
\end{array}
\end{equation*}
\end{linenomath}
for every \((\phi,\psi)\in X_r\) and 
\begin{linenomath}\begin{equation*}
\begin{split}
\langle E'' (u_0)(\phi,\psi),(h,k)\rangle =&\, 
(\phi',h')_2 + (\psi,k)_2 + m^2 (\phi,h)_2 \\
+&\, (G''(|\phi_0|)\text{Re}(\phi),h)_2 + 
\big(|\phi_0|^{-1} G'(|\phi_0|)\,i\text{Im}(\phi),h\big)_2
\end{split}
\end{equation*}\end{linenomath}
and
\begin{linenomath}\begin{equation*}
\langle Q''(u_0)(\phi,\psi),(h,k)\rangle = (\phi,ik)_2 + (h,i\psi)_2.
\end{equation*}\end{linenomath}
\end{enumerate}
\end{proposition}
In (i) and (ii) and (iv), the proof follows from the application of ideas illustrated in 
\cite[Theorem~2.6]{AP93}. The quoted theorem proves
\(C(L^p(\Omega),L^{\alpha p}(\Omega))\) regularity, with \(\Omega\) bounded domain of \(\mathbb{R}^n\).
However, their technique can be adapted to our setting, by taking advantage
of the bounded inclusion \(H^1 _r\subseteq L^\infty\). 
In (iii), roughly speaking, 
the multiplication by \(a\) allows to reduce to bounded domains of \(\R\), where 
it is known that the resolvent of the Laplacian is compact.
In the remainder of this section, we check that Assumptions 1, 2 and 3 of \cite{GSS87}.
\subsection*{Assumption~1}
According to \cite[\S3]{GV89}, \eqref{eq.NLKG} is
locally well-posed in \(X_r\), meaning that for every initial
datum \((\phi_0,\psi_0)\) there exists a unique solution 
\(
\phi\colon [0,T)\times\mathbb{R}\to\mathbb{C}
\)
to \eqref{eq.NLKG} such that
\begin{linenomath}\begin{equation*}
\phi(t,\cdot)\in C_t H\sp 1 _x ([0,T)\times\R)\cap 
C\sp 1 _t L\sp 2 _x ([0,T)\times\R).
\end{equation*}\end{linenomath}
Moreover, \(E\) and \(Q\) are constant on the trajectory
\((\phi(t,\cdot),\partial_t\phi(t,\cdot))\).
\begin{proposition}
\label{prop.ground-state}
Given \(a,b\) and \(m\) positive real numbers, there exists \(\omega_* > 0\) and one-parameter family 
\(u_\omega := (R_\omega,-i\omega R_\omega)\) such that 
\begin{enumerate}[\ \ (a).]
\item \((\omega_*,m)\ni\omega\mapsto u_\omega\in X_r\)
\item \(R_\omega\) is positive, symmetric-decreasing and \(xR_\omega'(x) < 0\)
unless \(x = 0\).
\end{enumerate}
\end{proposition}
\begin{proof}
We define a function \(R_\omega\) on an open interval of \(\mathbb{R}\) containing the origin
as the solution to the initial value problem
\begin{linenomath}\begin{equation}\label{eq.IVP}
\begin{split}
&R_\omega ''(x) - G'(R_\omega(x)) - (m^2 - \omega^2) R_\omega(x) = 0,\\
&R_\omega (0) = R_*(\omega),\quad R_\omega'(0) = 0
\end{split}
\end{equation}\end{linenomath}
where \(R_*(\omega)\) is the first positive solution to the equation 
\begin{linenomath}\begin{equation}
\label{eq.86}
V(s) := -\frac{2G(s)}{s^2} = m^2 - \omega^2.
\end{equation}\end{linenomath}
From \cite[Theorem~5]{BL83a}, such solution extends to a positive strictly
symmetric-decreasing function on \(\mathbb{R}\). Moreover, from 
\cite[Remark~6.3]{BL83a} the functions \(R_\omega,R_\omega'\) and \(R_\omega''\)
have exponential decay, implying that \(R_\omega\) is in 
\(H^2 _r(\mathbb{R},\mathbb{R})\).
In order to ensure that \(R_*(\omega)\) exists, we restrict to \(\omega > \omega_*\),
\begin{linenomath}\begin{equation*}
\omega_* := (m^2 - \sup(V))^{\frac{1}{2}} = \sqrt{m^2 - \frac{a^2}{2b}}
\end{equation*}\end{linenomath}
which is the smallest \(\omega\) such that \eqref{eq.86} has at least one solution.
We define \(u_\omega := (R_\omega,-i\omega R_\omega)\). 
\end{proof}
\begin{lemma}
\label{lem.family}
The function 
\(u\colon (\omega_*,m)\to X_r\)
is \(C^1 ((\omega_*,m),X_r)\).
\end{lemma}
\begin{proof}
We can rely on the argument of \cite[Lemma~20]{SS85}, provided adaptation
to the spatial dimension \(N = 1\) is done. We fix \(\omega_0\in (\omega_*,m)\), 
$ R_0 := R_{\omega_0} $. 
The equation \eqref{eq.IVP} can be 
rewritten as
\(
R_0 = -\mathscr{R}(\lambda_0)G'(R_0)
\)
where \(\lambda_0 := \omega_0^2 - m^2\).
Therefore, it is convenient to define the function 
\begin{linenomath}
\begin{equation*}
\begin{array}{c}
\mathscr{G}\colon (\omega_*^2 - m^2,0)\times H^1 _r (\mathbb{R},\mathbb{R})
\to H^1 _r (\mathbb{R},\mathbb{R}) \\[0.7em]
\mathscr{G}(\lambda,\phi) := \phi + \mathscr{R}(\lambda)G' (\phi).
\end{array}
\end{equation*}
\end{linenomath}
The proof of the regularity
of \(R_\omega\) takes several steps.

(i). \(\mathscr{G}\) is well defined. Since \(G'\) is \(C^1\) and \(G'(0) = 0\), the function
\(G'(\phi)\) is in \(H^1 _r\) from (i) of Proposition~\ref{prop.nem}.
Since \(\mathscr{R}(\lambda) H^1 _r\subseteq H^1 _r\), we have 
\(\phi + \mathscr{R}(\lambda)G' (\phi)\in H^1 _r\). 
(ii).  It is differentiable at every point \((\lambda_0,\phi_0)\). 
More precisely,
\begin{linenomath}\begin{equation*}
\langle \mathscr{G}'(\lambda_0,\phi_0),(\lambda,\phi)\rangle = 
\phi + \lambda\mathscr{R}(\lambda_0)^2 G'(\phi_0)
+\mathscr{R}(\lambda_0)G''(\phi_0)\phi.
\end{equation*}\end{linenomath}
In fact,
\begin{linenomath}\begin{equation*}
\begin{split}
&\,\mathscr{G}(\lambda_0 + \lambda,\phi_0 + \phi)\\
=&\,\phi_0 + \phi + (\mathscr{R}(\lambda_0) + 
\lambda\mathscr{R}(\lambda_0)^2 + o(\lambda))(G'(\phi_0 + \phi))\\
=&\,\phi_0 + \phi + (\mathscr{R}(\lambda_0) + 
\lambda\mathscr{R}(\lambda_0)^2 + o(\lambda))(G'(\phi_0) + G''(\phi_0)\phi + o(\phi))\\
=&\,\mathscr{G}(\lambda_0,\phi_0) + \phi + \lambda\mathscr{R}(\lambda_0)^2 G'(\phi_0) + \mathscr{R}(\lambda_0)G''(\phi_0)\phi
+f
\end{split}
\end{equation*}\end{linenomath}
where
\begin{linenomath}\begin{equation*}
f = o(\lambda)G'(\phi_0 + \phi)
+ \lambda\mathscr{R}(\lambda_0)^2 (G''(\phi_0)\phi + o(\phi)) + \mathscr{R}(\lambda_0)o(\phi).
\end{equation*}\end{linenomath}
The first equality follows from \cite[p.~174,\ Theorem~6.7~of~\S{III}]{Kat95} and the
resolvent equation, \cite[p.~36,\ Eq.~(5.5)~of~\S{I}]{Kat95}. Since \(G'\) is \(C^2\),
and \(G'(0)=G''(0) = 0\), the second equality follows from (ii) of
Proposition~\ref{prop.nem}, that is the notation \(o(\phi)\) applies
in the sense of the \(H^1\) norm. From \(\mathscr{R}(\lambda_0)\in\mathscr{L}(H^1_r,H^1_r)\)
and (iii) of Proposition~\ref{prop.nem}, \(f\)
is \(o(\lambda,\phi)\).
(iii). \(\mathscr{G}\) is in \(C^1 ((\omega_*^2 - m^2,0)\times H^1_r,H^1_r)\). This follows
from
the continuity of the resolvent, \cite[Theorem~6.7~of~\S{III}]{Kat95}, and 
(i) of Proposition~\ref{prop.nem}. 

(iv). \(\partial_\phi\mathscr{G}(\omega_0^2 - m^2,R_0)\in GL(H^1_r)\). Here the conclusions apply
to a specific point in \(H^1_r\), namely \(R_0\), and the compactness result of (iii) of Proposition~\ref{prop.nem} 
is used for the first time. Since \(G'(0) = G''(0) = 0\)  and \(R_0\) decays to zero exponentially the operator
\begin{linenomath}\begin{equation*}
\partial_\phi\mathscr{G} (\omega_0^2 - m^2,R_0) = \phi+\mathscr{R}(\lambda_0)
G''(R_0)\phi
\end{equation*}\end{linenomath}
is a compact perturbation of 
\(I_{H^1_r}\), by (iii) of Proposition~\ref{prop.nem}. Therefore, it is a
Fredholm operator of index zero, by \cite[p.~238,\ Theorem~5.26~of~\S{IV}]{Kat95}.
We can show that \(\partial_\phi\mathscr{G}(\omega_0^2 - m^2,R_0)\) is injective. 
In fact, given
\(\phi\in H^1_r\) such that
\begin{linenomath}\begin{equation*}
\phi+\mathscr{R}(\omega_0^2 - m^2)G''(R_0)\phi = 0,
\end{equation*}\end{linenomath}
there holds 
\begin{linenomath}\begin{equation}
\label{eq.family.5}
-\phi'' + (m^2 - \omega_0^2)\phi + G''(R_0)\phi = 0.
\end{equation}\end{linenomath}
In \(H^1 (\mathbb{R};\mathbb{C})\), the solutions to the elliptic equation above are multiples of \(R_0'\). 
However, \(R_0'\) is an
odd function. Therefore, \(\phi\equiv 0\). Since the Fredholm index is zero, the operator is also surjective. Therefore, it is invertible.

Conclusions (i-iv) allow us to apply the Implicit Function Theorem. There exists
an open interval \(\omega_0^2 - m^2\in J_0\subseteq (\omega_*^2 -m^2,0)\)
and \(\phi\colon J_0\to H^1_r\) such that  
\(
\mathscr{G}(\lambda,\phi_\lambda) = 0
\)
for every \(\lambda\in J_0\), 
and \(\phi_{\omega_0^2 - m^2} = R_0\).
Therefore, 
\begin{linenomath}\begin{equation}
\label{eq.family.4}
-\phi_\lambda'' -\lambda\phi_\lambda - G'(\phi_\lambda) = 0.
\end{equation}\end{linenomath}
To conclude we prove that \(\phi_{\omega^2 - m^2}\) coincides with \(R_\omega\): 
the former is regular, \(R_\omega\) will as also be regular.
There exists \(\delta > 0\) such
that for 
\(\lambda\in (\omega_0^2 - m^2 - \delta,\omega_0^2 - m^2 + \delta)\cap J_0\),
there holds \(\phi_\lambda (0) > 0\), because \(\phi_{\omega_0^2 - m^2} (0) = R_0 (0) > 0\).
From \cite[Theorem~5]{BL83a},
there exists only one even, positive solution, vanishing at infinity to \eqref{eq.family.4}.
Therefore, 
\(\phi_{\omega^2 - m^2} = R_{\omega}\).
Thus, \(R_\omega\) is \(C^1 ((\omega_*,m),H^1_r)\) and
\(u_\omega\in C^1 ((\omega_*,m),X_r)\) as claimed. In fact, since
the second component is a multiple of the first, the regularity is
\(C^1 ((\omega_*,m),H^1 _r\times H^1_r)\).
\end{proof}
\subsection*{Assumption 2}
From (iv) of Proposition~\ref{prop.nem}, \((E - \omega Q)'(u_\omega) = 0\)
for every \(\omega\in(\omega_*,m)\). From 
Proposition~\ref{lem.family}, \(u\in C^1((\omega_*,m),X_r)\), giving
(a) and (b) of \cite[Assumption~2]{GSS87}. Items (c) and (d) of the same assumptions
are summarized in the next proposition.
\begin{proposition}
\(u_\omega\in D(T'(0)^3)\cap D(JIT'(0)^2)\) and \(T'(0)u_\omega\neq 0\).
\end{proposition}
\begin{proof}
Since \(T'(0)\) is bounded, \(D(T'(0)^3) = X_r\). Also, for every \(\omega\in (\omega_*,m)\),
there holds
\(
T'(0)^2 u_\omega = (-R_\omega,i\omega R_\omega)
\). 
Since \(R_\omega\) is \(H^2 _r\), we have 
\begin{linenomath}\begin{equation*}
\begin{split}
(T'(0)^2 u_\omega,(h,k))_{X_r} &= (-R_\omega',h')_2 + (-R_\omega,h)_2 + (i\omega R_\omega,k)_2 \\
&=
(R_\omega'' - R_\omega,h)_2 + (i\omega R_\omega,k)_2 \\
&= I_0 (R_\omega'' - R_\omega,i\omega R_\omega).
\end{split}
\end{equation*}\end{linenomath}
Since \(R_\omega\in H^2 _r\), \(I(T'(0)^2 u_\omega)\in I_0(L^2_r\times H^1_r)\). Therefore, \(I(T'(0)^2 u_\omega)\in D(J)\),
proving the first of the two statements. The second one follows from 
\(T'(0)u_\omega = (-iR_\omega,-\omega R_\omega)\neq 0\).
\end{proof}
\begin{remark}
Although it is clear from the definition, we wish to stress that 
the operator \(I\) appearing in \cite[Assumption~2]{GSS87} and in the proposition above is the Riesz representation of linear functionals
defined in \eqref{eq.Riesz} and \textsl{not} \(I_0\) defined in 
\eqref{eq.shatah-identification}.
\end{remark}
\begin{remark}
When \(\omega = \omega_*\), the constant function \(R_{\omega_*}\equiv R_* (\omega_*) = (2b)^{-1} a\) solves the initial value problem \eqref{eq.IVP}, but
it is not a square integrable function. When \(\omega = m\),
\(R_* (m) = b^{-1}a\). This the "zero-mass" of the problem \eqref{eq.IVP},
whose existence in \(H^1_r\cap C^2\) 
is guaranteed by \cite[\S5]{BL83a}. However, since
\(R_* (\omega) = R_\omega (0)\to 0\) as \(\omega\to m\), there is no convergence
\(R_\omega\to R_m\) in \(H^1 _r\). In conclusion, \((\omega_*,m)\) is
a maximal interval of existence of a regular one-parameter family.
\end{remark}
\section{The spectrum of the Hessian}
In this section, we prove that \cite[Assumption~3]{GSS87} is satisfied. We consider the Hessian of \(d\) at \(\omega_0\), defined as the bounded operator
\(H\in\mathscr{L}(X_r,X_r)\) such that 
\begin{linenomath}\begin{equation*}
H(\phi,\psi) = I^{-1}((E'' - \omega_0 Q'')(\phi,\psi)),
\end{equation*}\end{linenomath}
where \(I\) is the Riesz isomorphism from \(X_r\) to \(X_r^*\).
The result of Sturm-Liouville Theory of \cite[p.\ 228,\ \S10.4]{Inc44} with the following extension:
if the assumption that \(u\) has two zeroes is replaced by the assumption that
\(u,v\in C^2 _0 (\R)\), then one can still conclude that \(v\) has at least one zero
in \((-\infty,+\infty)\).
\begin{lemma}
\label{lem.essential-spectrum}
\(H\) is a self-adjoint, bounded and Fredholm operator of index zero
on \(X_r\).
\end{lemma}
\begin{proof}
It is convenient to have an explicit expression of the Hessian. 
Given \((\phi,\psi)\in X_r\), 
we set \(H(\phi,\psi) =: (H_1 (\phi,\psi),H_2 (\phi,\psi))\).
Using \((0,k)\in X_r\) as test vector in the formula for the second derivative in Proposition~\ref{prop.nem},
we obtain
\begin{linenomath}\begin{equation*}
H_2 (\phi,\psi) = \psi + i\omega_0\phi.
\end{equation*}\end{linenomath}
Using \((h,0)\) with \(h\in H^2 _r (\R,\C)\), the equality
\begin{linenomath}\begin{equation*}
-H_1'' + m^2 H_1 = -\phi'' + m^2\phi + G''(R_0)\text{Re}(\phi) + iR_0^{-1} G'(R_0)\text{Im}(\phi) - \omega_0 i\psi
\end{equation*}\end{linenomath}
follows. Applying \(\mathscr{R}(-m^2)\) to both sides of the equality, one obtains
\begin{linenomath}\begin{equation*}
H_1 (\phi,\psi) = \phi + 
\mathscr{R}(-m^2)\big(G''(R_0)\text{Re}(\phi) + 
iR_0^{-1} G'(R_0)\text{Im}(\phi) - i\omega_0\psi\big).
\end{equation*}\end{linenomath}
Since \(H^2_r\times L^2_r\) is a dense subset of \(X_r\), the two equalities for \(H_1\) and \(H_2\) hold
on \(X_r\).
Since \(R_0\) vanishes at infinity, both \(G''(R_0)\) and
\(iR_0^{-1} G'(R_0)\) are \(C^1 _0 (\R,\C)\). Therefore, the Hessian is a compact perturbation of the operator
\begin{linenomath}\begin{equation*}
A(\phi,\psi) := 
\big(\phi - \mathscr{R}(-m^2)(i\omega_0\psi),i\omega_0\phi + \psi
\big)
\end{equation*}\end{linenomath}
by (iii) of Proposition~\ref{prop.nem}.
\(A\) is in \(GL(X_r)\). Given \((\alpha,\beta)\in X_r\) the equation \(A(\phi,\psi) = (\alpha,\beta)\) can be solved as follows:
the second component reads \(\psi = \beta - i\omega_0\phi\). A substitution in the first component gives
\begin{linenomath}\begin{equation*}
(I_{H^1_r} - \omega_0^2\mathscr{R}(-m^2))\phi = \alpha + \omega_0 \mathscr{R}(-m^2)i\beta.
\end{equation*}\end{linenomath}
The operator on the lefthand-side is bounded and invertible in \(H^1_r\). In fact, by merely checking
operators composition, through the resolvent equation \cite[p.~36]{Kat95} one can deduce that the inverse is 
\(I_{H^1_r} + \omega_0^2\mathscr{R}(\omega_0^2 -m^2)\). Therefore,
\begin{linenomath}\begin{equation*}
\phi = (I_{H^1_r} + \omega_0^2\mathscr{R}(\omega_0^2 -m^2))\big(\alpha + \omega_0 \mathscr{R}(-m^2)i\beta\big)
\end{equation*}\end{linenomath}
and
\begin{linenomath}\begin{equation*}
\psi = \beta - i\omega_0(I_{H^1_r} + \omega_0^2\mathscr{R}(\omega_0^2 -m^2))\big(\alpha + \omega_0 \mathscr{R}(-m^2)i\beta\big). 
\end{equation*}\end{linenomath}
Since \(A\in GL(X_r)\), it is a Fredholm operator of index zero. Since \(H\) is a compact perturbation of a Fredholm operator of index zero,
by \cite[Theorem~5.26~of~\S{IV}]{Kat95}, 
it is Fredholm operator with index zero.
\end{proof}
\begin{theorem}
For every \(\omega\in (\omega_*,m)\), 
\begin{enumerate}[(a).]
\item the kernel of \(H\) is spanned by \(T'(0)u_\omega\)
\item the operator \(H\) has exactly one negative simple eigenvalue
\item the rest of its spectrum is positive and bounded away from zero.
\end{enumerate}
\end{theorem}
\begin{proof}
\(T'(0)u_\omega\in\ker (H)\) follows from the remarks preceding \cite[Eq.~(2.18)]{GSS87}. Now, given an element \((\phi,\psi)\) in the kernel of the Hessian,
from (iv) of Proposition~\ref{prop.nem}, we have  
\begin{linenomath}\begin{equation*}
0 = \langle (E - \omega_0 Q)''(\phi,\psi),(0,k)\rangle = 
(\psi,k)_2  - \omega_0 (\phi,ik)_2
\end{equation*}\end{linenomath}
for every \(k\in L^2 _r\). Therefore, 
\begin{linenomath}\begin{equation*}
\psi = -i\omega_0\phi.
\end{equation*}\end{linenomath}
Now we apply the second derivative to 
\((ih,0)\in X_r\) for every \textsl{real-valued} function \(h\in H^1 _r\). 
From 
\begin{linenomath}\begin{equation*}
\text{Re}(\phi\overline{ih}) = \text{Im}(\phi) h,\quad
\text{Re}(\phi'\overline{ih'}) = \text{Im}(\phi)' h',
\end{equation*}\end{linenomath}
it follows
\begin{linenomath}\begin{equation}
\label{eq.family.3}
\begin{split}
0&=\langle (E - \omega_0 Q)''(\phi,\psi),(ih,0)\rangle \\
&= (\text{Im}(\phi)',h')_2 + m^2(\text{Im}(\phi),h)_2 + 
(R_0^{-1} G'(R_0)\text{Im}(\phi),h)_2 - \omega_0 (ih,i\psi)_2\\
&= 
(\text{Im}(\phi)',h')_2 + m^2(\text{Im}(\phi),h)_2 + 
(R_0^{-1} G'(R_0)\text{Im}(\phi),h)_2 
- \omega_0 (ih,\omega_0\phi)_2 \\
&= 
(\text{Im}(\phi)',h')_2 + m^2(\text{Im}(\phi),h)_2 + (R_0^{-1} G'(R_0)\text{Im}(\phi),h)_2 - \omega_0^2 (\text{Im}(\phi),h)_2 \\
&= 
(L_- (\text{Im}(\phi)),h)_2.
\end{split}
\end{equation}\end{linenomath}
\(L_-\) is the unbounded operator with domain \(H^2 _r\subseteq L^2_r\) defined as
\begin{linenomath}\begin{equation*}
L_- (f) = -f'' + R_0^{-1} G'(R_0)f + (m^2 - \omega_0^2) f.
\end{equation*}\end{linenomath}
Therefore, \(L_- (\text{Im}(\phi)) = 0\).
The kernel of \(L_-\) has dimension one. Since \(R_0\) is in \(\ker(L_-)\), there 
exists \(\mu\in\mathbb{R}\) such that
\(\text{Im}(\phi) = \mu R_0\). Now, we apply the second derivative to 
\((h,0)\in X_r\) for every function \(h\in H^1 (\mathbb{R},\mathbb{R})\).
Therefore,
\begin{linenomath}\begin{equation*}
\begin{split}
0&=\langle (E - \omega_0 Q)''(\phi,\psi),(h,0)\rangle \\
&= (\text{Re}(\phi)',h')_2 + m^2(\text{Re}(\phi),h)_2 + (G''(R_0)\text{Re}(\phi),h)_2 - 
\omega_0 (h,i\psi)_2\\
&= 
(\text{Re}(\phi)',h')_2 + m^2(\text{Re}(\phi),h)_2 + 
(G''(R_0)\text{Re}(\phi),h)_2 - \omega_0 (h,i\cdot (-i\omega_0\phi))_2 \\
&= 
(\text{Re}(\phi)',h')_2 + m^2(\text{Re}(\phi),h)_2 + 
(G''(R_0)\text{Re}(\phi),h)_2 - 
\omega_0^2 (\text{Re}(\phi),h)_2 \\
&= 
(L_+ (\text{Re}(\phi)),h)_2,
\end{split}
\end{equation*}\end{linenomath}
where \(L_+\) is the unbounded operator with domain 
\(H^2\subseteq L^2\) defined as
\begin{linenomath}\begin{equation*}
L_+ (f) = -f'' + G''(R_0)f + (m^2 - \omega_0^2) f.
\end{equation*}\end{linenomath}
Therefore, \(L_+ (\text{Re}(\phi)) = 0\). From the remarks right
after \eqref{eq.family.5} it follows \(\text{Re}(\phi) = 0\).
In conclusion,
\begin{linenomath}\begin{equation*}
(\phi,\psi) = \mu\left(iR_0,\omega_0 R_0\right) = -\mu(-iR_0,-\omega_0R_0)
= -\mu T'(0) u_\omega. 
\end{equation*}\end{linenomath}
(b). From the \cite[Proposition~4.2]{Wei86} and \cite[p.~187]{GSS87},
the operator \(L_+\) has exactly one negative, simple eigenvalue.
We use the notation \(-\alpha_{\omega_0}^2\) and 
\(\chi_{\omega_0}\in H^2 _r (\mathbb{R},\mathbb{R})\)
for this eigenvalue and the corresponding eigenvector. 
There holds
\begin{linenomath}\begin{equation*}
\begin{split}
(H(\chi_{\omega_0},-i\omega_0\chi_{\omega_0}),(\chi_{\omega_0},-i\omega_0\chi_{\omega_0}))_{X_r}  = 
-\alpha_{\omega_0}^2 \|\chi_{\omega_0}\|^2_2 < 0.
\end{split}
\end{equation*}\end{linenomath} 
Therefore, there exists at least one negative eigenvalue. We prove that 
this eigenvalue is unique. Let \((\phi,\psi)\) be an eigenvector
with eigenvalue \(\lambda < 0\), that is
\begin{linenomath}\begin{equation*}
H(\phi,\psi) = \lambda(\phi,\psi).
\end{equation*}\end{linenomath}
Taking the \(X_r\) inner product with \((0,k)\), we obtain 
\begin{linenomath}\begin{equation*}
(\psi,k)_2 - \omega_0 (\phi,ik)_2 = \lambda (\psi,k)_2
\end{equation*}\end{linenomath}
for every \(k\in L^2_r\). Then
\begin{linenomath}\begin{equation}
\label{eq.family.6}
\psi = \frac{i\omega_0}{\lambda - 1}\phi.
\end{equation}\end{linenomath}
Taking the \(X_r\) inner product with \((i\text{Im}(\phi),0)\), 
from \eqref{eq.family.3}  
we obtain
\begin{linenomath}\begin{equation*}
\lambda\|\text{Im}(\phi)'\|_2^2 + m^2\lambda\|\text{Im}(\phi)\|_2^2 =  
(L_- (\text{Im}({\phi})),\text{Im}({\phi}))_2
\end{equation*}\end{linenomath}
implying
\(
(L_-(\text{Im}(\phi),\text{Im}(\phi)))_2\leq 0
\).
From \cite[\S3]{Wei86} \(L_-\geq 0\). Therefore, 
\(\text{Im}(\phi) = 0\), because \(\lambda < 0\).
Therefore, the imaginary part of the first component of every eigenvector is zero. 
We take the inner product with \((h,0)\), where 
\(h\in H^1_r (\mathbb{R},\mathbb{R})\). Then
\begin{linenomath}\begin{equation*}
\begin{split}
\lambda (\phi',h')_2+ m^2\lambda(\phi,h)_2 &= 
(H(\phi,\psi),(h,0))_{X_r} \\
&= (\phi',h') + m^2 (\phi,h)_2 + (G''(R_0)\phi,h)_2 - \omega_0 (h,i\psi)_2 \\
&= (\phi',h') + m^2 (\phi,h)_2 + (G''(R_0)\phi,h)_2 \\
&- \omega_0 (\lambda - 1)^{-1} (h,i\cdot i\omega_0\phi)_2 \\
&= (\phi',h') + m^2 (\phi,h)_2 + (G''(R_0)\phi,h)_2 \\
&+\, \omega_0^2(\lambda - 1)^{-1}
(\phi,h)_2.
\end{split}
\end{equation*}\end{linenomath}
Therefore, \(\phi\) satisfies
the second order differential equation
\begin{linenomath}\begin{equation*}
K_\lambda \phi'' - F''(R_0)\phi - G_\lambda\phi = 0
\end{equation*}\end{linenomath}
where 
\begin{linenomath}\begin{equation*}
K_\lambda := 1 - \lambda,\quad
G_\lambda := -\left(\frac{\omega_0^2}{1 - \lambda} + m^2\lambda\right).
\end{equation*}\end{linenomath}
Suppose that
there are two eigenvectors \((\phi_1,\psi_1)\) and \((\phi_2,\psi_2)\)
corresponding to negative eigenvalues \(\lambda_1 \leq\lambda_2 < 0\).
Clearly, \(1 - \lambda_1\geq 1 - \lambda_2 > 1\). That is, 
\(K_{\lambda_1}\geq K_{\lambda_2} > K_0\). Since
\begin{linenomath}\begin{equation*}
\frac{dG_\lambda}{d\lambda} = -\frac{\omega_0^2}{(1 - \lambda)^2} - m^2 < 0
\end{equation*}\end{linenomath}
we also have \(G_{\lambda_1}\geq G_{\lambda_2} > G_0 := -\omega_0^2\).
Now, suppose that \(\lambda_1 < \lambda_2\), that is 
\(K_{\lambda_2} < K_{\lambda_1}\). Then from 
\cite[p.\ 228,\ \S10.4]{Inc44}, 
\(\phi_2\) has a zero in \(x_*\in (-\infty,+\infty)\). Since \(\phi_2\) is
even, \(|x_*|\) is also a zero of \(\phi_2\).
Taking the derivative in \eqref{eq.IVP}, we obtain
\begin{linenomath}\begin{equation*}
(R_0')'' - F''(R_0) R_0' + \omega_0^2 R_0' = 0.
\end{equation*}\end{linenomath}
Since \(\lambda_2 < 0\), we have \(K_{\lambda_2} > K_0\) and 
\(G_{\lambda_2} > G_0\). Again, from \cite[p.\ 228,\ \S10.4]{Inc44}, \(R_0'\) has a zero \(x_{**}\in (|x_*|,+\infty)\).
This contradicts (b) of Proposition~\ref{prop.ground-state}, according
to which \(R_0'(x) < 0\) for every \(x > 0\). 
Therefore, \(\lambda_1 = \lambda_2 =: \lambda\) and 
\(\phi_1 = k\phi_2\) for some \(k\in\mathbb{R}\). From \eqref{eq.family.6},
\begin{linenomath}\begin{equation*}
\begin{split}
(\phi_1,\psi_1) &= \left(\phi_1,\frac{i\omega_0}{\lambda - 1}\phi_1\right) = 
\left(k\phi_2,k\frac{i\omega_0}{\lambda - 1}\phi_2\right) \\
&= 
k\left(\phi_2,\frac{i\omega_0}{\lambda - 1}\phi_2\right) = k(\phi_2,\psi_2)
\end{split}
\end{equation*}\end{linenomath}
showing that the unique negative eigenvalue is also simple. We denote it
by \(\lambda_{\omega_0}\).

(c). Since \(\lambda_{\omega_0}\) is the unique negative eigenvalue,
the complement of \(\{0,\lambda_{\omega_0}\}\) in \(\sigma(H)\)
is positive. From Lemma~\ref{lem.essential-spectrum}, \(H\) is a self-adjoint Fredholm operator. 
From \cite[Lemma]{Phi96}, there exists
\(a > 0\) such that \(\sigma(H)\cap (-a,a)\) is a finite set of eigenvalues. Then,
the essential spectrum is bounded away from the origin.
\end{proof}
\section{Stability and instability}
\label{ch.4}
Since \cite[Assumption~3]{GSS87} is satisfied, by \cite[Theorem~2]{GSS87}, stability and instability
of \((R_0,-i\omega_0 R_0)\)
relies on the study of the convexity of the function \(d\) in a neighbourhood
of \(\omega_0\). We have 
\begin{linenomath}\begin{equation*}
\begin{split}
d'(\omega) &= (E - \omega Q)'(u_\omega) - Q(u_\omega) = -Q(u_\omega) \\
&= -(i\cdot -i\omega R_\omega,R_\omega)_2 = -\omega\|R_\omega\|_2^2 =: -\sigma(\omega).
\end{split}
\end{equation*}\end{linenomath}
The third equality follows from \eqref{eq.q}.
Then, we will inspect the sign of the derivative \(\sigma\). Calculations of the next proof 
substantially rely on the assumption that \(G\) is a cubic-quartic non-linearity, as defined
in \eqref{eq.cubic-quartic}.
\begin{proof}[Proof~of~Theorem~\ref{thm.stability-and-instability}]
We divide the proof in two steps. In the first one, we evaluate \(d'\).
In the second, we study the sign of \(d''\). 

\textit{First step}. We multiply \eqref{eq.IVP} by \(2R_\omega'\) and integrate. Since
\(R_\omega\) vanishes at infinity, we have
\begin{linenomath}\begin{equation*}
R_\omega '(x)^2 = (m^2 - \omega^2) R_\omega (x) ^2 + 2G(R_\omega(x)).
\end{equation*}\end{linenomath}
Since $ R_\omega' < 0 $, we can write
\begin{linenomath}\begin{equation*}
R_\omega '(x) = -\sqrt{(m^2 - \omega^2) R_\omega (x)^2 + 2G(R_\omega (x))}.
\end{equation*}\end{linenomath}
From the remarks preceding this proof, we have
\begin{linenomath}\begin{equation}
\label{eq.84}
\begin{split}
-d'(\omega) = \sigma(\omega) &= 2\omega\int_0^\infty  R_\omega (x)^2 dx \\
&= 2\omega \int_0^\infty
\frac{R_\omega (x)^2 R_\omega'(x)dx}{-\sqrt{(m^2 - \omega^2) R_\omega (x)^2 + 2G(R_\omega(x))}}\\
&= 2\omega \int_0^\infty
\frac{R_\omega (x) R_\omega'(x)dx}{-\sqrt{m^2 - \omega^2 - 2a R_\omega(x) + 2bR_\omega(x)^2}}\\
&= \omega\int_{0}^{R_*(\omega)} \frac{sds}%
{\sqrt{m^2 - \omega^2 - 2as + 2bs^2}} \\
&= \omega\int_{0}^{R_*(\omega)} \frac{\left(s - \frac{a}{2b}\right)ds}%
{\sqrt{m^2 - \omega^2 - 2as + 2bs^2}} \\
&+ \frac{\omega a}{2b}\int_{0}^{R_*(\omega)} \frac{ds}{\sqrt{m^2 - \omega^2 - 2as + 2bs^2}} =: A + B.
\end{split}
\end{equation}\end{linenomath}
The fifth equality follows from the substitution \(R_\omega (x) = s\) on the interval
\((0,+\infty)\). Since \(R_* (\omega)\) is the first positive zero to 
\(m^2 - \omega^2 = 2aR_* (\omega) - 2bR_* (\omega)^2\), there holds
\begin{linenomath}\begin{equation}
\label{eq.85}
\left|R_* (\omega) - \frac{a}{2b}\right|^2  = 
\frac{a^2}{4b^2} - \frac{m^2 - \omega^2}{2b} = \frac{a^2}{4b^2}(1 - \alpha^2(\omega))
\end{equation}\end{linenomath}
where
\begin{linenomath}\begin{equation}
\label{eq.117}
\alpha(\omega) = \left(\frac{2b(m^2 - \omega^2)}{a^2}\right)^{\frac{1}{2}}.
\end{equation}\end{linenomath}
In order to find a suitable integration by substitution, we 
rearrange the argument of the square root
\begin{linenomath}\begin{equation}
\label{eq.116}
\begin{split}
&\,m^2 - \omega^2 - 2as + 2bs^2 \\
=&\,2aR_* (\omega) - 2bR_* (\omega)^2 - 2as + 2bs^2\\
=&\,2b\left(\bigg(s - \frac{a}{2b}\bigg)^2 - \bigg(R_*(\omega) - \frac{a}{2b}\bigg)^2\right).
\end{split}
\end{equation}\end{linenomath}
From \eqref{eq.116}, we obtain
\begin{linenomath}\begin{equation*}
\begin{split}
A &= \frac{\omega}{\sqrt{2b}}\int_{0}^{R_*(\omega)} \frac{\left(s - \frac{a}{2b}\right)ds}%
{\sqrt{(s - \frac{a}{2b})^2 - 
(R_*(\omega) - \frac{a}{2b})^2}} \\
&= \frac{\omega}{\sqrt{2b}}\left[\sqrt{\left(s - \frac{a}{2b}\right)^2 - 
\left(R_*(\omega) - \frac{a}{2b}\right)^2}\right]_0^{R_*(\omega)} \\
&= -\frac{\omega}{\sqrt{2b}}\sqrt{\frac{a^2}{4b^2} - 
\left(R_*(\omega) - \frac{a}{2b}\right)^2} \\
&= 
-\frac{\omega}{\sqrt{2b}}\sqrt{\frac{a^2}{4b^2} - \frac{a^2}{4b^2}(1 - \alpha^2(\omega))} =
-\frac{\omega a}{(2b)^{\frac{3}{2}}}\alpha(\omega)
\end{split}
\end{equation*}\end{linenomath}
and
\begin{linenomath}\begin{equation*}
\begin{split}
B &=\frac{\omega a}{(2b)^{\frac{3}{2}}}
\int_0^{R_*(\omega)}\frac{ds}{\sqrt{(s - \frac{a}{2b})^2 - 
(R_*(\omega) - \frac{a}{2b})^2}} \\
&=\frac{\omega a}{(2b)^{\frac{3}{2}}}
\Bigg[\ln\bigg|s - \frac{a}{2b} + \sqrt{\left(s - \frac{a}{2b}\right)^2 - 
\left(R_*(\omega) - \frac{a}{2b}\right)^2}\bigg|\Bigg]_0^{R_* (\omega)}\\
&=\frac{\omega a}{(2b)^{\frac{3}{2}}} \ln\bigg|R_*(\omega) - \frac{a}{2b}\bigg|
- \frac{\omega a}{(2b)^{\frac{3}{2}}} 
\ln\bigg| - \frac{a}{2b} + \sqrt{\frac{a^2}{4b^2} - 
\left(R_*(\omega) - \frac{a}{2b}\right)^2}\bigg|\\
&=\frac{\omega a}{(2b)^{\frac{3}{2}}}\ln\bigg|\frac{a}{2b}\sqrt{1 - \alpha^2(\omega)}\bigg|
-\frac{\omega a}{(2b)^{\frac{3}{2}}} \ln\bigg| - \frac{a}{2b} + \frac{a}{2b}\alpha(\omega)\bigg|\\
&= \frac{\omega a}{2(2b)^{\frac{3}{2}}}\ln(1 - \alpha^2(\omega))
-\frac{\omega a}{2(2b)^{\frac{3}{2}}} \ln(1 - \alpha(\omega))^2 \\
&= 
\frac{\omega a}{2(2b)^{\frac{3}{2}}}\ln\left(\frac{1 + \alpha(\omega)}{1 - \alpha(\omega)}\right).
\end{split}
\end{equation*}\end{linenomath}
From \eqref{eq.84}, 
\begin{linenomath}\begin{equation*}
\sigma(\omega) = A + B = -\frac{\omega a}{(2b)^{\frac{3}{2}}}\alpha(\omega) + 
\frac{\omega a}{2(2b)^{\frac{3}{2}}}\ln\left(\frac{1 + \alpha(\omega)}{1 - \alpha(\omega)}\right).
\end{equation*}\end{linenomath}
\textit{Second step}. In order to study the sign of the derivative of $ \sigma $, we represent it as
the composite function of $ \alpha $, which is a strictly decreasing and surjective function from
the interval $ (\omega_*,m) $ to $ (0,1) $. One can check that \(\alpha(\omega_*) = m\)
directly from \eqref{eq.85}. From the definition of \(\alpha\), we have 
\begin{linenomath}\begin{equation*}
\omega = \frac{a}{\sqrt{2b}}\sqrt{\tau - \alpha(\omega)^2}.
\end{equation*}\end{linenomath}
Therefore,
\begin{linenomath}\begin{equation}
\label{eq.119}
\sigma(\omega) = \frac{a^2}{8b^2} k_1 (\alpha(\omega)),\quad\sigma'(\omega) = 
\frac{a^2}{8b^2} k_1'(\alpha(\omega))\alpha'(\omega)
\end{equation}\end{linenomath}
where
\begin{linenomath}\begin{equation*}
k_1 (\alpha) = \sqrt{\tau - \alpha^2}\left(
\ln\left(\frac{1 + \alpha}{1 - \alpha}\right)- 2\alpha\right),\quad\alpha\in (0,1).
\end{equation*}\end{linenomath}
We have
\begin{linenomath}\begin{equation*}
\begin{split}
k_1 '(\alpha) = 
-\frac{\alpha}{\sqrt{\tau - \alpha^2}}
\ln\left(\frac{1 + \alpha}{1 - \alpha}\right) + 
\frac{2\alpha^2\sqrt{\tau - \alpha^2}}{1 - \alpha^2}.
\end{split}
\end{equation*}\end{linenomath}
Then
\begin{linenomath}\begin{equation*}
\begin{split}
k_1 '(\alpha)\sqrt{\tau - \alpha^2}\left(\frac{1 - \alpha^2}{2\alpha^2}\right) 
&= -\frac{1 - \alpha^2}{2\alpha}\ln\left(\frac{1 + \alpha}{1 - \alpha}\right)
+ \tau - \alpha^2\\
&= \tau - \left[\frac{1 - \alpha^2}{2\alpha}\ln\left(\frac{1 + \alpha}{1 - \alpha}\right) + \alpha^2\right] =: \tau - k_2(\alpha).
\end{split}
\end{equation*}\end{linenomath}
From \eqref{eq.119} it follows
\begin{linenomath}\begin{equation}
\label{eq.120}
\sigma'(\omega) = \frac{a^2}{4b^2}\cdot\frac{\alpha^2(1 - \alpha^2)\alpha'(\omega)}%
{\sqrt{\tau - \alpha(\omega)^2}} 
\big(\tau - k_2(\alpha(\omega))\big).
\end{equation}\end{linenomath}
We define
\begin{linenomath}\begin{equation*}
\tau_* := \sup_{\alpha\in(0,1)} k_2(\alpha).
\end{equation*}\end{linenomath}
The behaviour of $ k_2 $ at the endpoints is \(\lim_{\alpha\to 0} k_2(0) = 1\) and \(\lim_{\alpha\to 1} k_2 = 1\) proving
that \(\tau_*\geq 1\). Also,
\begin{linenomath}\begin{equation*}
\begin{split}
k_2'(\alpha) &= -\frac{1}{2}\left(1+\frac{1}{\alpha ^2}\right)\ln\left(\frac{1 + \alpha}{1 - \alpha}\right) + \frac{1}{\alpha} + 2\alpha
\end{split}
\end{equation*}\end{linenomath}
Since \(\lim_{\alpha\to 1} k_2'(\alpha) = -\infty\), we can infer that \(\tau_* > 1\).

\textit{Conclusions}

The case $ \tau\geq\tau_* $.
Suppose that constants \(a,b,m\) defining \(G\) in \eqref{eq.cubic-quartic} as such that \(\tau(a,b,m) > \tau_* = \sup(k_2)\).
Since $ \tau > k_2(\alpha(\omega)) $ on $ (\omega_*,m) $, from \eqref{eq.120} $ \sigma' $
is negative on $ (\omega_*,m) $. Therefore \(d''(\omega) > 0\). From \cite[Theorem~2]{GSS87}, the orbit \eqref{eq.orbit} is stable
for every \(\omega\in (\omega_*,m)\). When \(\tau(a,b,m) = \tau_* = \sup(k_2)\), \(d''\) has at least one zero, as
\(\tau_* > \max\{k_2 (0),k_2 (1)\}\), implying that \(\sup(k_2)\) is achieved in the interior of \(0\leq\alpha\leq 1\).
We can show that \(d''\) has exactly one zero. In fact, 
\begin{linenomath}\begin{equation*}
\alpha^2 k_2' = f_1 + f_2,\quad f_1 := -\frac{1}{2}(1+\alpha ^2)\ln\left(\frac{1 + \alpha}{1 - \alpha}\right),\quad
f_2 := \alpha + 2\alpha^3.
\end{equation*}\end{linenomath}
Since \(f_1\) is a strictly monotonically decreasing function and \(f_2\) 
is a strictly monotonically increasing function, \(\alpha^2 k_2'\) has at most one zero on \((0,1)\).
Then \(k_2'\) has at most one zero on \((0,1)\), which is the maximum point of \(k_2\).
Set \(\alpha_d :=\text{argmax}(k_2)\).
Since \(k_2(\alpha) < \tau_*\) for \(\alpha\neq\alpha_d\), there holds \(d''(\omega) > 0\)
unless \(\omega = \omega_d := \alpha^{-1}(\alpha_d)\). From \cite[Theorem~2]{GSS87} the orbit \eqref{eq.orbit} is stable for
every \(\omega\), as \(d\) is strictly convex in a neighbourhood of \(\omega_d\).

The case $ 1 < \tau < \tau_* $.
If the non-linearity \eqref{eq.cubic-quartic} satisfies \(\tau(a,b,m) < \tau_*\), cases of instability may occur. Since \(1 < \tau\), 
there are \textsl{two} values \(\alpha_1(\tau) > \alpha_2(\tau)\) such that \(k_2 (\alpha_1) = k_2(\alpha_2) = 0\)
and \(k_2 < \tau\) on \((0,\alpha_2)\cup (\alpha_1,1)\) and and \(k_2 > \tau\) on \(\alpha_2,\alpha_1\). Therefore,
\eqref{eq.orbit} is stable on 
\begin{linenomath}\begin{equation*}
(\omega_*,\omega_1 (\tau))\cup(\omega_2 (\tau),m)
\end{equation*}\end{linenomath}
and unstable on 
\begin{linenomath}\begin{equation*}
\omega_1 (\tau)\leq\omega\leq\omega_2 (\tau)
\end{equation*}\end{linenomath}
where \(\omega_i (\tau) = \alpha^{-1}(\alpha_i (\tau))\) for \(i=1,2\).

The case $ 0 < \tau\leq 1 $.
This is the case where stability and instability subsets of \((\omega_*,m)\) are both connected. If \(\tau > 0\), there
is a unique \(\alpha_3 = \alpha(\tau)\) such that \(k_2(\alpha_3) = \tau\) and \(k_2(\alpha_3) < \tau\) if
\(\alpha < \alpha_3\) and \(k_2 (\alpha_3) > \tau\) if \(\alpha > \alpha_3\). Therefore,
from \cite[Theorem~2]{GSS87}
there exists \(\omega_3 := \alpha^{-1}(\alpha_3)\) such that \eqref{eq.orbit} is stable on
\begin{linenomath}\begin{equation*}
(\omega_*,\omega_3(\tau))
\end{equation*}\end{linenomath} 
and unstable on 
\begin{linenomath}\begin{equation*}
\omega_3 (\tau)\leq\omega < m.
\end{equation*}\end{linenomath}
\end{proof}
The case \(\tau = 0\) does not occur due to the restrictions set on \(a,b\)
and \(m\). If we allowed \(\tau = 0\), we would obtain a quadratic Klein-Gordon
equation, as the non-linearity is the pure power \(-as^2\). This case has been covered in the calculations of
\cite[p.~325]{Sha83} which apply to the one-dimensional cases as well, even if the
author set the restriction \(n\geq 3\) in the introduction of the paper. We highlight the sharp change in the stability scenario, as in the cases (ii) and (iii)
of Theorem~1, orbits become unstable as \(\omega\) increases, while in the
pure-power case, orbits become stable as \(\omega\) increases.
\def\cprime{$'$} \def\cprime{$'$} \def\cprime{$'$} \def\cprime{$'$}
  \def\cprime{$'$} \def\cprime{$'$} \def\cprime{$'$} \def\cprime{$'$}
  \def\cprime{$'$} \def\polhk#1{\setbox0=\hbox{#1}{\ooalign{\hidewidth
  \lower1.5ex\hbox{`}\hidewidth\crcr\unhbox0}}} \def\cprime{$'$}
  \def\cprime{$'$} \def\cprime{$'$}
\providecommand{\bysame}{\leavevmode\hbox to3em{\hrulefill}\thinspace}
\providecommand{\MR}{\relax\ifhmode\unskip\space\fi MR }
\providecommand{\MRhref}[2]{%
  \href{http://www.ams.org/mathscinet-getitem?mr=#1}{#2}
}
\providecommand{\href}[2]{#2}

\end{document}